\documentclass{amsart}

\usepackage[T1]{fontenc}
\usepackage{amssymb}
\usepackage[all]{xy}
\usepackage[hidelinks]{hyperref}

\newtheorem{theorem}{Theorem}
\newtheorem{corollary}{Corollary}
\newtheorem{lemma}{Lemma}
\newtheorem{proposition}{Proposition}
\theoremstyle{definition}
\newtheorem{example}{Example}
\newtheorem*{maintheorem}{Main Theorem}

\newcommand{\KK}{\mathbb{K}}
\newcommand{\up}{{\uparrow}}
\newcommand{\cP}{\mathcal{P}}

\title[Superdecomposable Algebras]{Superdecomposable Incidence and Complete Path Algebras}

\author[C. Hat\.{i}po\u{g}lu]{Can Hat\.{i}po\u{g}lu}
\address{College of Engineering and Technology, American University of the Middle East, Kuwait}
\email{osman.hatipoglu@aum.edu.kw}

\author[C. Lomp]{Christian Lomp}
\address{Department of Mathematics of the Faculty of Science and Center of Mathematics, University of Porto, Rua Campo Alegre, 687, 4169-007 Porto, Portugal}
\email{clomp@fc.up.pt}
\thanks{The second author was partially supported by Centro de Matemática da Universidade do Porto, member of LASI, which is financed by
national funds through FCT – Fundação para a Ciência e a Technology, I.P., under the project with reference UID/00144/2025, doi: \url{https://doi.org/10.54499/UID/00144/2025}.}

\date{}

\begin{document}

\begin{abstract}
Generalizing the setting of the second author's work in \cite{Lomp2007}, we characterize quivers $Q$ for which the dual algebra $R$ of the path coalgebra of $Q$, equivalently the complete path algebra, has no nonzero uniform right ideals: this happens exactly when every vertex of $Q$ can reach a vertex emitting at least two distinct arrows. Equivalently, the injective hull $E(R)$ of the right regular module is superdecomposable. These conditions are also equivalent to the path coalgebra having no nonzero hollow factor modules as a left module over its dual algebra. We also characterize nonempty locally finite posets whose incidence algebras have superdecomposable injective hulls of their right regular modules: this happens exactly when no principal up-set is upward directed.
\end{abstract}

\maketitle

\section{Introduction} 

In \cite{Lomp2007}, the second author constructed an indecomposable module without nonzero hollow factor modules, thereby showing that hollow-lifting modules need not be lifting. The existence of such modules had been left open in \cite{Orhan-Keskin-Tribak}. The construction starts with a field $\KK$ and a finite alphabet $\Sigma$ containing at least two letters. Then we let $\Sigma^*$ denote the free monoid of words in $\Sigma$, including the empty word. The space $R$ of all functions $\Sigma^*\to\KK$ becomes a $\KK$-algebra under pointwise addition and the convolution product for $f,g \in R$:
\[(fg)(w)=\sum_{uv=w}f(u)g(v), \qquad \forall w\in \Sigma^*\]
Let $M$ be the $\KK$-vector space with basis $\Sigma^*$. For $f\in R$ and $w\in \Sigma^*$, the action
$f\cdot w=\sum_{uv=w}f(v)u$, 
extended $\KK$-linearly, makes $M$ a left $R$-module. As shown in \cite[\S~2]{Lomp2007}, this module is indecomposable and has no nonzero hollow factor modules.

A module-theoretic explanation of this example is provided by \cite[Proposition~3.1]{Lomp2007}: if a semi-injective coretractable module has a nonzero hollow factor module, then its endomorphism ring contains a nonzero uniform right ideal. Here and throughout, endomorphisms are written as acting on the right, with multiplication defined by $m(fg)=(mf)g$. 

The module $M$ above is the path coalgebra of the quiver with a single vertex and one loop for each letter of $\Sigma$, and $R$ is its dual algebra. 
More generally, every coalgebra $C$ over $\KK$ is a self-injective self-cogenerator (hence semi-injective and coretractable) as a left module over its dual algebra $C^*$. With our convention for endomorphism rings, there is a $\KK$-algebra isomorphism
\[C^*\cong\operatorname{End}_{C^*}({}_{C^*}C).\]
Consequently, if $C^*$ contains no nonzero uniform right ideals, then $C$, regarded as a left $C^*$-module, has no nonzero hollow factor modules.

Motivated by this observation, we study the existence of uniform right ideals in the dual algebras of path and incidence coalgebras. Our main result characterizes the absence of nonzero hollow factor modules for path coalgebras in terms of the branching structure of the underlying quiver. We call a vertex \emph{branching} if it emits at least two distinct arrows. A vertex has access to another vertex if there is a directed path from the former to the latter, allowing paths of length zero.

\begin{maintheorem}
Let $\KK$ be a field, let $Q$ be a nonempty quiver, and let
$C=\KK Q$ be its path coalgebra with dual algebra $R=C^*$.
The following conditions are equivalent:

\begin{enumerate}
\item Every vertex of $Q$ has access to a branching vertex.
\item $R$ has no nonzero uniform right ideals.
\item The injective hull $E(R_R)$ is superdecomposable.
\item $C$ has no nonzero hollow factor modules as a left
      $R$-module under the canonical action.
\end{enumerate}
\end{maintheorem}
The theorem is proved in Section~\ref{sec:Path Coalgebras and Uniform Right Ideals}.

In Section~\ref{sec:incidence-algebras}, we consider the absence of nonzero uniform right ideals in incidence algebras of nonempty locally finite posets: this happens exactly when no principal up-set is upward directed. We also give examples which show that superdecomposability of the injective hull of the regular module is not left-right symmetric.

\section{Hollow Factors, Uniform Submodules, and Superdecomposable Modules}\label{sec:superdecomposable modules}
A nonzero module $M$ is called \emph{superdecomposable} if it does not have nonzero indecomposable direct summands. A nonzero module is \emph{uniform} if any two nonzero submodules have a nonzero intersection and it is called \emph{hollow} if it cannot be written as the sum of two proper submodules. For a nonzero module $M$, its injective hull $E(M)$ is superdecomposable if and only if $M$ contains no nonzero uniform submodules. Indeed, if $D$ is an indecomposable direct summand of $E(M)$, then it is injective and hence uniform. Since $M$ is essential in $E(M)$ and since nonzero submodules of uniform modules are uniform, $D\cap M$ is a nonzero uniform submodule of $M$. Conversely, if $U$ is a nonzero uniform submodule of $E(M)$, then $E(U)$ is uniform. Since an injective module is uniform if and only if it is indecomposable, it follows that $E(U)$ is (isomorphic to) an indecomposable direct summand of $E(M)$ and $E(M)$ is not superdecomposable.

\begin{example}
Let $R = \KK^{\mathbb{N}}/I$ be the quotient of the ring of all sequences in $\KK$ by the ideal $I = \KK^{(\mathbb{N})}$ of all sequences with finitely many nonzero entries. Given any nonzero element $a = (a_n) + I$ with $\operatorname{supp}((a_n)) = \{n \in \mathbb{N} : a_n \neq 0 \}$, we can write $\operatorname{supp}((a_n)) = X \cup Y$ into a disjoint union of infinite sets. Let $e_X$ be the idempotent of $R$ given by $(e_X)_n = 1$ if and only if $n \in X$ and zero otherwise. Then $ae_X R$ and $ae_Y R$ are both nonzero such that $ae_X R \cap ae_Y R = 0$. Thus, $E(R_R)$ is superdecomposable.
\end{example}
\begin{example}\label{powerseries}
Let $R = \KK \langle x,y\rangle$ be the free algebra in two variables over a field $\KK$. If $I$ is a nonzero right ideal of $R$ and $f$ is a nonzero element of $I$, then $fxR$ and $fyR$ are two nonzero right ideals contained in $I$ with zero intersection. Hence, no nonzero right ideal of $R$ is uniform and its injective hull $E(R_R)$ is superdecomposable. More generally, the same argument holds for 
$D=\KK\langle\langle x,y\rangle\rangle$, the free power series ring in $x$ and $y$. 
\end{example}
We record here for convenience the aforementioned result in \cite{Lomp2007}. Recall that a left $R$-module $M$ is \emph{semi-injective} if for any factor module $N$ of $M$, monomorphism $k:N\to M$ and homomorphism $g:N\to M$, there exists $h:M\to M$, such that $kh = g$. Equivalently, a left $R$-module $M$ with endomorphism ring $S = \operatorname{End}_R(M)$ is semi-injective if $\operatorname{Ann}_S (\ker f) = fS$ for every $f \in S$, where $\operatorname{Ann}_S (N) = \{f \in S : Nf = 0\}$. Also, $M$ is called \emph{coretractable} if $\operatorname{Hom}_R (M/N, M) \neq 0$ for every proper submodule $N$ of $M$. 
\begin{lemma}[{\cite[Proposition 3.1]{Lomp2007}}]
Let $M$ be a nonzero semi-injective coretractable module with endomorphism ring $S=\mathrm{End}_R (M)$. If $N<M$ and $M/N$  is hollow, then $\operatorname{Ann}_S(N)$ is a nonzero uniform right ideal of $S$.
\end{lemma}
\begin{corollary}\label{E(S) superdecomposable iff no hollow factors}
Let $M$ be a nonzero semi-injective coretractable module with endomorphism ring $S=\mathrm{End}_R(M)$. Then $E(S_S)$ is superdecomposable if and only if $M$ has no nonzero hollow factor modules.
\end{corollary}
\begin{proof}
($\Rightarrow$) If $M/N$ is a nonzero hollow factor, then $\operatorname{Ann}_S(N)$ is a nonzero uniform right ideal. So $E(S_S)$ is not superdecomposable by the equivalence given at the beginning of Section~\ref{sec:superdecomposable modules}. 

($\Leftarrow$) Suppose that $M$ has no nonzero hollow factor modules and that $I$ is a nonzero uniform right ideal of $S$. Let $0 \neq f \in I$. Then the right ideal $fS \subseteq I$ is also uniform. Since $M$ is semi-injective, we have $\operatorname{Ann}_S (\ker f) = fS$. 

We claim that $M/\ker f$ is hollow. Otherwise, there exist proper submodules $A,B$ of $M$ both containing $\ker f$ such that $A+B = M$. Since $M$ is coretractable, the sets $\operatorname{Hom}_R (M/A,M)$ and $\operatorname{Hom}_R(M/B,M)$ are nonzero and hence $\operatorname{Ann}_S (A)$ and $\operatorname{Ann}_S (B)$ are nonzero right ideals contained in $\operatorname{Ann}_S (\ker f) = fS$. Since $\operatorname{Ann}_S (A) \cap \operatorname{Ann}_S (B) = \operatorname{Ann}_S (A+B) = \operatorname{Ann}_S (M) = 0$, this contradicts the uniformity of $fS$. Thus, $M/\ker f$ is a nonzero hollow factor module, contradicting the hypothesis. Hence $S$ has no nonzero uniform right ideals, so $E(S_S)$ is superdecomposable.
\end{proof}

The above result can also be stated for coalgebras. Let $C$ be a coalgebra over a field $\KK$ and $C^* = \operatorname{Hom}_\KK (C, \KK)$ be its dual algebra. $C$ is canonically a left $C^*$-module with the action
\[f\cdot c = \sum_{(c)} f(c_2)c_1\]
for all $f \in C^*$ and $c \in C$. From \cite[9.1]{Brzezinski-Wisbauer}, we know that $C$ is an injective cogenerator in the category of right $C$-comodules. In particular, $C$ is a self-injective, self-cogenerator $C^*$-module. Also from \cite[3.12]{Brzezinski-Wisbauer}, we know that $\operatorname{End}_{C^*}(C) \to C^*$ given by $f\mapsto \epsilon \circ f$ is an isomorphism of $\KK$-algebras. So, the results above apply and we have the following version of Corollary~\ref{E(S) superdecomposable iff no hollow factors} for coalgebras:
\begin{corollary}\label{cor:coalgebra-hollow}
Let $C$ be a nonzero coalgebra over $\KK$. Then $E(C^*_{C^*})$ is superdecomposable if and only if $C$ has no nonzero hollow factor modules as a left $C^*$-module, equivalently, no nonzero hollow factors as a right $C$-comodule. 
\end{corollary}
\section{Path Coalgebras and Uniform Right Ideals}\label{sec:Path Coalgebras and Uniform Right Ideals}
Let $\KK$ be a field and $Q = (Q_0, Q_1, s, t)$ be a quiver with $Q_0 \ne \emptyset$ and source and target maps $s,t: Q_1 \to Q_0$. A path of length $n \ge 1$ is a list $p = a_1 a_2 \cdots a_n$ of arrows $a_i$ such that $t(a_i) = s(a_{i+1})$ for all $1 \leq i < n$. We write $|p|$ for the length of a path $p$. Elements $e \in Q_0$ are considered as paths of length zero. Denote the set of all paths by $\cP$ and extend the maps $s$ and $t$ from $Q_1$ to $\cP$ by defining $s(p) = s(a_1)$ and $t(p) = t(a_n)$. For a vertex $e \in Q_0$, we set $s(e) = e = t(e)$. 

Let $C = \KK Q$ be the linear span of all paths in $\cP$. The product $pq$ of two paths $p$ and $q$ is defined when $t(p) = s(q)$. The prefix order defined by
\[p \leq q \iff q = pu \text{ for some path } u \in \cP\]
is a partial order on $\cP$. Moreover, 
$C$ is a coalgebra with comultiplication
\[\Delta(p) = \sum_{uv=p}u \otimes v\]
for all $p \in \cP$ and counit $\epsilon(e) = 1$ for every vertex $e \in Q_0$ and $\epsilon(p) = 0$ for every $p \in \cP \setminus Q_0$. The dual space $R = C^* = \operatorname{Hom}_\KK (C, \KK)$ becomes a $\KK$-algebra by
\[(fg)(p) = \sum_{uv = p}f(u)g(v) \]
for all $f, g \in R$ and a path $p$. The identity element is the counit $\epsilon$ defined above. For each path $p \in \cP$ there exists a functional $\pi_p$ defined by $\pi_p (p) = 1$ and $\pi_p (q) = 0$ for $q \neq p$. Note that for $p \in \cP$ with $e = s(p)$ and $e' = t(p)$, one has $\pi_e \pi_p = \pi_p = \pi_p \pi_{e'}$. In particular, for every vertex $e$, the right ideal $\pi_e R$ consists of functions that are supported in paths starting at $e$. The algebra $R$ is called the \emph{complete path algebra} of $Q$. There is an embedding given by $p \mapsto \pi_p$ of the ordinary path algebra in $R$ as the subalgebra of finite linear combinations of paths. Moreover, $R$ is its completion with respect to the topology whose basic neighbourhoods of zero are the cofinite monomial ideals, see \cite{Iovanov}. 

For the  alphabet $\Sigma = \{x,y\}$, one could interpret the monoid of words $\Sigma^*$ in \cite{Lomp2007} as the set of paths that one could obtain in the following quiver:
\[\xymatrix{
  \bullet \ar@(ur,dr)[]^y \ar@(dl,ul)[]^x.
}\]
The complete path algebra of $Q$ is  isomorphic to the free power series ring  $R=\KK\langle\langle x,y\rangle\rangle$ in $x$ and $y$, which turns the free algebra $C=\KK\langle x,y\rangle$ into a left $R$-module by $$\sum_{z\in \Sigma^*} \lambda_z z \cdot w = \sum_{w=uv} \lambda_v u,$$ for any $w\in \Sigma^*$. Since $R$ has no nonzero uniform right ideals, $C$ has no nonzero hollow factor module by Corollary \ref{cor:coalgebra-hollow}. Hence, given a quiver $Q=(Q_0,Q_1,s,t)$ and its path coalgebra $C$, it is natural to ask when the dual algebra $C^*$ has no nonzero uniform right ideals. 

We now state and prove some preliminary results which will lead to the proof of the main theorem. Denote by $S_e = \{ p \in \cP : s(p) = e\}$, resp. $T_e = \{ p \in \cP : t(p) = e\}$, the paths starting resp. terminating at $e \in Q_0$. Note that $S_e \cap T_e$ consists of all closed paths starting and terminating at $e$. 
\begin{lemma}\label{Finite and Totally Ordered Path Sets}
Let $Q = (Q_0, Q_1, s, t)$ be a nonempty quiver, $\KK$ be a field, and $R$ be the dual of the path coalgebra of $Q$. Let $e \in Q_0$. 
\begin{enumerate}
\item If $S_e$ is finite, then $\pi_e R$ contains a nonzero minimal right ideal.
\item If $S_e$ is totally ordered with respect to the prefix order, then $\pi_e R$ contains a nonzero uniserial Noetherian right ideal.
\end{enumerate}
Analogous statements hold for $T_e$ and the left ideal $R\pi_e$ with the suffix order.
\end{lemma}
\begin{proof}
(1) Consider first the case $S_e = \{e\}$. Then for any path $q$,
\[(\pi_e f)(q) = \sum_{uv=q} \pi_e(u)f(v) = \begin{cases}f(e) & \text{ if } q = e\\ 0 & \text{ else}\end{cases}\]
because if $q = ev$, then $s(q) = e$ and hence $q \in S_e = \{e\}$. Hence, $\pi_e f = \lambda \pi_e$ for $\lambda = f(e)$, and $\pi_e R = \pi_e\KK$ is one-dimensional and therefore minimal. In case $S_e$ is finite, choose a path $p \in S_e$ of maximal length. Then $S_{t(p)} = \{t(p)\}$ and $\pi_e R \supseteq \pi_e \pi_p R = \pi_p \KK$ is minimal. 

(2) If $S_e$ is finite, then the result follows from (1). Suppose that $S_e$ is an infinite set which is totally ordered with respect to the prefix order. Hence, for any $i \geq 0$, there exists precisely one path $p_i \in S_e$ of length $i$. Clearly, $p_0 = e$ and there exists a sequence $\{a_i\}_{i\geq 1} \in Q_{1}^{\mathbb{N}}$ such that $p_i = a_1\cdots a_i$, i.e. $S_e = \{e\} \cup \{p_i : i\geq 1\}$. Set
\[u_{i,i} := t(p_i), \quad u_{k,m} = a_{k+1}\cdots a_m, \quad \forall m > k.\]
Then for any $m \geq k$ one has $p_m = p_k u_{k,m}$. 

Let $0 \neq f \in \pi_e R$ and let $n$ be minimal with $f(p_{n})\ne 0$. Set $\lambda = f(p_{n})^{-1}$. Define $g \in R$ to vanish outside the set $\{u_{n,m}: m \ge n\}$, and on these paths by recursion on the length:
\[g(u_{n,n}) = \lambda, \qquad g(u_{n,m}) = -\lambda\sum_{k = n + 1}^m f(p_k)g(u_{k,m}) \quad (m > n).\]
We claim that $fg = \pi_{p_n}$, because
\[(fg)(p_n) = \sum_{uv=p_n} f(u)g(v) = f(p_n)g(u_{n,n}) = 1,\]
since $f(p_k) = 0$ for $k < n$ because of the minimality of $n$. Moreover, for any $m>n$, 
\[ (fg)(p_m) = \sum_{k=n}^{m} f(p_k)g(u_{k,m}) = \lambda^{-1}g(u_{n,m}) + \sum_{k=n+1}^m f(p_k)g(u_{k,m}) = 0. \]
It is also clear that $(fg)(p_k) = 0$ for $k < n$ by the minimality of $n$. Thus, $fg = \pi_{p_n}$ and  $\pi_{p_n}R \subseteq fR$. On the other hand, $f \in \pi_{p_n}R$ because $\pi_{p_n}R$ consists of functions that are supported on paths which have $p_n$ as a prefix and $f$ is such a function by the assumption. Hence, $fR = \pi_{p_n}R$. 

Let $J$ be a nonzero right subideal of $\pi_e R$. From the preceding paragraph, for each nonzero $f\in J$, there exists $i \ge 0$ such that $fR = \pi_{p_i}R$. Let $k = \min \{i : \text{there exists } f\in J \text{ with } f(p_i)\ne 0 \}$. Then $J = \pi_{p_k}R$ and each nonzero subideal of $\pi_e R$ is of the form $\pi_{p_i}R$ and hence $\pi_e R$ is uniserial and Noetherian since a strictly ascending chain of these ideals would give a strictly decreasing sequence of their indices. 
\end{proof}
\begin{corollary}
Let $Q = (Q_0, Q_1, s, t)$ be a quiver with convolution algebra $R$ as above. If $E(R_R)$ is superdecomposable, then for any $e \in Q_0$, $S_e$ is an infinite set containing at least two incomparable paths.
\end{corollary}
Recall that we call a vertex $e$ in $Q$ \emph{branching} if there are two distinct arrows whose source is $e$. 
\begin{lemma}\label{Access to a Branching Vertex}
Let $\KK$ be a field, $Q$ be a quiver, and $C$ and $R$ be its path coalgebra and convolution algebra, respectively. Suppose that $f$ is a nonzero element in $R$ and $p$ is a path of minimal length in the support of $f$ so that $f(\alpha) = 0$ for every path $\alpha$ such that $|\alpha| < |p|$. If $t(p)$ is a branching vertex, then the right ideal $I = fR$ is not uniform. 
\end{lemma}
\begin{proof}
Let $v = t(p)$. By hypothesis, there are distinct arrows $a,b$ whose source is $v$. Consider the right ideals
\[U = f \pi_a  R, \qquad V = f \pi_b  R\]
contained in $I$, which are both nonzero since $(f \pi_a)(pa) = f(p) \neq 0$ and $(f \pi_b)(pb) = f(p)\neq 0$.

We claim that $U \cap V = 0$. Suppose that this is not true and that
\[0 \neq f \pi_a  g = f \pi_b  h\]
for some $g, h \in R$. Since $f \pi_a  g$ is nonzero, there exists a path $u$ for which
\[(f \pi_a  g)(u) = \sum_{u_1 au_2 = u}f(u_1)g(u_2) \neq 0.\]
So there exists a decomposition $u = u_1 ar$ with $f(u_1)g(r) \neq 0$ (otherwise the above sum would be zero); in particular, the set of paths $\rho$ with $s(\rho) = t(a)$ and $g(\rho)\neq 0$ is nonempty. Among all such paths $\rho$, choose $r$ of minimal length. Similarly, from $f \pi_b  h \neq 0$ we deduce that the set of paths $\sigma$ with $s(\sigma) = t(b)$ and $h(\sigma)\neq 0$ is nonempty, and we choose $s$ of minimal length among such paths. Without loss of generality (interchanging the pairs $(a,g)$ and $(b,h)$ if necessary), we may assume
\[|r| \leq |s|.\]

We evaluate $f \pi_a  g$ and $f \pi_b  h$ at $par$. Let
\[(f \pi_a  g)(par) = \sum_{\alpha a\gamma = par}f(\alpha)g(\gamma).\]
For every decomposition in the above sum, exactly one of the following holds:
\[|\alpha| < |p|, \qquad |\alpha| = |p|, \qquad |\alpha| > |p|.\]
In the first case $f(\alpha) = 0$, in the third $g(\gamma) = 0$, and in the second case we necessarily have $(\alpha,\gamma) = (p,r)$. Hence 
\[(f \pi_a  g)(par) = f(p)g(r) \neq 0.\]
On the other hand, when calculating $(f \pi_b   h)(par)$, then
\[(f \pi_b  h)(par) = \sum_{\alpha b\gamma = par}f(\alpha)h(\gamma).\]
If  $|\alpha| < |p|$, then $\alpha$ is a prefix of $p$ and $f(\alpha)=0$ by the minimality of $p$.
If  $|\alpha| > |p|$, then $pa$ is a prefix of $\alpha$. Thus, $\alpha b \gamma = par$ yields $|\gamma| < |r| \le |s|$ and hence $h(\gamma)=0$ by minimality of  $s$. The only remaining case where $|\alpha| = |p|$ is impossible because the arrow immediately following $p$ in $par$ is $a \neq b$. Therefore,
\[(f \pi_b  h)(par) = 0,\]
contradicting $f \pi_a  g = f \pi_b  h$. Therefore, $U\cap V = 0$. Hence,  $I$ contains two nonzero right ideals with zero intersection and it is not uniform.
\end{proof}
Recall that a vertex $v$ of a quiver has access to another vertex $v'$ if there is a path $p$ with $s(p) = v$ and $t(p)=v'$. Now we are ready to prove the main theorem:
\begin{theorem}\label{Main Theorem}
Let $\KK$ be a field, $Q$ a nonempty quiver, and $C$ and $R$ be its path coalgebra and convolution algebra, respectively. The following statements are equivalent:
\begin{enumerate}
\item Every vertex of $Q$ has access to a branching vertex.
\item $R$ has no nonzero uniform right ideals.
\item $E(R_R)$ is superdecomposable.
\item $C$ has no nonzero hollow factor modules as a left
$R$-module.
\end{enumerate}
\end{theorem}
\begin{proof}
$(1) \Rightarrow (2)$ Let $I$ be a nonzero right ideal of $R$. Take $0 \neq f \in I$ and let $p$ be a path of minimal length such that $f(p) \neq 0$. By assumption, there exists a path $q$ such that $s(q) = t(p)$ for which $t(q)$ emits at least two distinct arrows. The functional 
\[\overline{f} = f\pi_q \in I,\]
satisfies $\overline{f}(pq) = f(p) \ne 0$ and has the path $pq$ as a minimal length element in its support, since $p$ is minimal and any element in its support must have $q$ as a suffix. Since $t(pq)$ is a branching vertex, then $\overline{f}R$ is not uniform by Lemma~\ref{Access to a Branching Vertex} which, in turn, implies $I$ is not uniform.

(2) $\Rightarrow$ (1) Suppose that $Q$ has a vertex $e$ which does not have access to a branching vertex. Then the set $S_e$ is totally ordered by the prefix order, because every vertex reachable from $e$ emits at most one arrow. Since minimal and uniserial ideals are uniform, in each case, it follows from Lemma~\ref{Finite and Totally Ordered Path Sets} that $R$ contains a nonzero uniform right ideal.

The equivalence of (2) and (3) is established at the beginning of Section~\ref{sec:superdecomposable modules}, while the equivalence of (3) and (4) follows from Corollary~\ref{cor:coalgebra-hollow}.
\end{proof}
The following example shows that for an arbitrary ring $R$, the property that the injective hull of the regular module is superdecomposable is not left-right symmetric.
\begin{example}
Consider the quiver
\[
Q: \xymatrix{
1 \ar[r]^z & 2 \ar@(ul,ur)[]^x \ar@(dl,dr)[]_y
}
\qquad \qquad 
Q^{op}: \xymatrix{
1  & 2 \ar[l]_z \ar@(ul,ur)[]^x \ar@(dl,dr)[]_y
}
\]
and let $C$ be its path coalgebra over a field $\KK$ and let $R=C^*$ be the complete path algebra of $Q$. By Theorem \ref{Main Theorem},  $E(R_R)$ is superdecomposable. We can identify $R$ with the generalized upper triangular matrix ring
$$R \simeq \left(\begin{array}{cc} \KK & D \\ 0 & D\end{array}\right), \qquad \mbox{ with } D=\KK\langle\langle x,y\rangle\rangle.$$
Note that $R$ has a minimal left ideal. Hence $E({_RR})$ is not superdecomposable. Alternatively, one could consider $R^{op}$ as the complete path algebra of the quiver $Q^{op}$, obtained from $Q$ by reversing arrows, which does not satisfy   Theorem \ref{Main Theorem}.
\end{example}
\begin{example}
Consider the quiver
\[
Q: \xymatrix{
1 \ar[r] &
2 \ar[r] &
3 \ar[r] &
4.
}
\]
Then $Q$ does not satisfy Theorem \ref{Main Theorem}. Adding to each arrow in $Q$ an arrow in the opposite direction, yields
\[
\overline{Q}:\xymatrix{
1 \ar@/^/[r] &
2 \ar@/^/[l] \ar@/^/[r] &
3 \ar@/^/[l] \ar@/^/[r] &
4 \ar@/^/[l]
}
\]
which satisfies the condition of Theorem \ref{Main Theorem}. More generally, let  $Q=(Q_0,Q_1,s,t)$ be any nonempty quiver. Create a new quiver  $\overline{Q}=(\overline{Q}_0,\overline{Q}_1,\overline{s},\overline{t})$, by adding to $Q_1$ a copy of each arrow in reverse direction, i.e. set $\overline{Q}_0=Q_0$, $\overline{Q}_1 = Q_1 \cup Q_1^{op}$, with $Q_1^{op}=\{ a^{op} : a\in Q_1\}$, $\overline{s}(a^{op})=t(a)$ and $\overline{t}(a^{op})=s(a)$, for any $a\in Q_1$, while $\overline{s}_{\mid_{Q_1}}=s$ and $\overline{t}_{\mid_{Q_1}}=t$. If $Q$ is a non-empty, connected  quiver $Q$, then  $\overline{Q}$ satisfies Theorem \ref{Main Theorem}, provided $Q$ has at least one arrow and is not the quiver $\xymatrix{ 1 \ar[r] &2 }$.
\end{example}

\section{Incidence Algebras}\label{sec:incidence-algebras}
We now apply the same module-theoretic criterion to incidence coalgebras. Let $(P, \le)$ be a partially ordered set. The set $\mathcal{I}$ of intervals is defined as 
\[\mathcal{I} = \{[a, b] \mid a \le b,\ a, b \in P \} \]
where $[a, b] = \{ c \in P \mid a \le c \le b\}$. A partially ordered set $P$ is called \emph{locally finite} if every interval $[a,b]$ is finite. 

Let $\mathbb{K}$ be a field and $(P, \le)$ be a locally finite poset. The incidence coalgebra $C$ of $P$ is the vector space whose basis is $\mathcal{I}$. That is, $C = \mathbb{K}\mathcal{I}$ with comultiplication
\[\Delta([a,b]) = \sum_{a\le c \le b} [a, c] \otimes [c, b].\]
 The counit is defined as $\epsilon([a,b]) = 0$ if $a\neq b$ and $\epsilon([a,b]) = 1$ if $a = b$. The dual algebra $C^*$ is called the \emph{incidence algebra} of $P$ and denoted by $I(P, \KK)$. Its multiplication is the convolution product which is defined by
 \[(fg)([a,b]) = \sum_{a\le c \le b}f([a,c])g([c,b])\]
 for all $f, g \in C^*$ and $a \leq b \in P$ and the counit is its identity element. 
 
We recall some terminology from order theory that will be needed in the characterization. Let $P$ be a partially ordered set. An \emph{up-set} of $P$ is a subset $U \subseteq P$ such that $a \in U$ and $a \le b$ imply $b \in U$. For $x \in P$, the set $\up x = \{y \in P \mid x \le y\}$ is the smallest up-set containing $x$,
called the \emph{principal up-set} of $x$. An element $u \in P$ is an \emph{upper bound} of a subset $S \subseteq P$ if $s \le u$ for all $s \in S$; an upper bound of $\{y,z\}$ is called a \emph{common upper bound} of $y$ and $z$. A nonempty subset $D \subseteq P$ is \emph{upward directed} if any two elements of $D$ have a common upper bound in $D$.

\begin{proposition}\label{proposition:Incidence Algebra Characterization}
Let $\mathbb{K}$ be a field, $(P, \le)$ a nonempty locally finite partially ordered set, and $A$ the incidence algebra of $P$. The following statements are equivalent:
\begin{enumerate}
\item The injective hull $E(A_A)$ is superdecomposable.
\item $A$ has no nonzero uniform right ideals.
\item For every $x \in P$, there exist $y,z \ge x$ having no common upper bound in $P$. Equivalently, no principal up-set $\up x = \{y \in P \mid x \le y \}$ is upward directed.
\end{enumerate}
These conditions are also equivalent to the incidence coalgebra having no nonzero hollow factor modules as a left module over its dual algebra.
\end{proposition}
\begin{proof}
The equivalence of (1) and (2) is established at the beginning of Section~\ref{sec:superdecomposable modules}. For the remaining equivalence, we let $e_{uv} \in A$ denote the dual basis element corresponding to $[u,v]$, i.e., it is the functional with $e_{uv}([x,y]) = \delta_{[u,v], [x,y]}$. It is easy to see that 
\[e_{uv}e_{st} = \begin{cases} e_{ut}, & v = s\\ 0, & v \neq s\end{cases}\]
$(3) \Rightarrow (2)$ Let $I$ be a nonzero right ideal of $A$ and choose $0 \neq f \in I$. Since $f$ is nonzero, there exist $a, x \in P$ such that $f([a,x]) \neq 0$. Choose $y,z \geq x$ with no common upper bound. Consider the right ideals 
\[U = (fe_{xy})A, \qquad V = (fe_{xz})A\]
which are both nonzero since $(fe_{xy})([a,y]) = f([a,x]) \neq 0$ and $(fe_{xz})([a,z]) = f([a,x]) \neq 0$. For any $h \in A$, we have
\[ (fe_{xy}h)([u,v]) = \begin{cases} f([u,x])h([y,v]), & u \le x,\ y \le v\\ 0, & \text{otherwise} \end{cases}\]
Thus, any element of $U$ is supported on intervals whose right endpoint lies in $\up y$. Similarly, every element of $V$ is supported on intervals whose right endpoint lies in $\up z$. Since $\up y \cap \up z = \emptyset$, we have $U \cap V = 0$. Therefore, $I$ is not uniform.

$(2) \Rightarrow (3)$ Suppose that there exists $x \in P$ such that $\up x$ is upward directed. We show that $e_{xx}A$ is a nonzero uniform right ideal of $A$. Let $U, V$ be nonzero right ideals of $A$ contained in $e_{xx}A$. Choose nonzero elements $f \in U$ and $g \in V$ and choose $y,z \geq x$ such that 
\[f([x,y]) \neq 0,\quad g([x,z]) \neq 0.\]
Since elements of $e_{xx}A$ are supported on intervals whose left endpoint is $x$, 
\[ fe_{yy} = f([x,y])e_{xy}, \quad ge_{zz} = g([x,z])e_{xz}. \]
Hence, $e_{xy} \in U$ and $e_{xz} \in V$. By assumption, there exists a common upper bound $t \geq y,z$. Then
\[0 \neq e_{xt} = e_{xy}e_{yt} = e_{xz}e_{zt} \in U \cap V.\]
Thus, $e_{xx}A$ is a nonzero uniform right ideal of $A$, contradicting (2). 
\end{proof}

The left-right asymmetry of the superdecomposability of the injective hull of the regular module can also be demonstrated for incidence algebras with the following example.
\begin{example}
Let $\Sigma = \{x,y\}$ be an alphabet with two letters and $P$ be the set of words in $\Sigma$, including the empty word $\varepsilon$. We order $P$ by the prefix order: 
\[u \leq v \iff \text{ $u$ is a prefix of $v$.}\]
If $u \leq v$, then the interval $[u,v]$ consists of the successive prefixes of $v$ starting at $u$ and thus it is a finite chain. In particular, $P$ is locally finite and its incidence coalgebra and incidence algebra $A = I(P, \KK)$ are defined. 

For an arbitrary word $w$, consider its two immediate extensions $wx$ and $wy$. A common upper bound would be a word having both $wx$ and $wy$ as prefixes. This is impossible since the next letter after $w$ would have to be both $x$ and $y$. Hence $\up (wx) \cap \up (wy) = \emptyset$. It follows from Proposition~\ref{proposition:Incidence Algebra Characterization} that $E(A_A)$ is superdecomposable. On the other hand, the injective hull of the left regular module $E(_A A)$ is not superdecomposable. To see this, observe that the empty word $\varepsilon$ is a least element in $P$. Then for any $f \in A$, the convolution product gives
\[fe_{\varepsilon \varepsilon} = f([\varepsilon, \varepsilon]) e_{\varepsilon \varepsilon}\]
because the only word $u$ such that $u \leq \varepsilon$ is $\varepsilon$ itself. Therefore, $Ae_{\varepsilon\varepsilon} = \KK e_{\varepsilon \varepsilon}$ is a one-dimensional, hence simple, left ideal. Let $S = Ae_{\varepsilon\varepsilon}$. Consequently, the injective hull $E(_A A)$ contains an indecomposable direct summand $E(S)$ and therefore it is not superdecomposable. 
\end{example}

\providecommand{\bysame}{\leavevmode\hbox to3em{\hrulefill}\thinspace}
\providecommand{\MR}{\relax\ifhmode\unskip\space\fi MR }
\providecommand{\MRhref}[2]{%
  \href{http://www.ams.org/mathscinet-getitem?mr=#1}{#2}
}
\providecommand{\href}[2]{#2}

\end{document}